\documentclass[a4paper]{amsart}
\usepackage[utf8]{inputenc}
\usepackage[T1]{fontenc}
\usepackage[symbol]{footmisc}
\usepackage{amssymb}
\usepackage{amsmath}
\usepackage{amsthm}
\usepackage[english]{babel}
\usepackage[noadjust]{cite}
\usepackage[hang]{caption}
\usepackage{tikz}
\usepackage{tikz-cd} 
\usepackage{enumerate}
\usepackage{mathrsfs}
\usepackage{xcolor}
\usepackage{hyperref}
\usepackage{mathrsfs}
\usepackage{physics}
\usepackage{bbm}
\usepackage[normalem]{ulem}

\newcommand{\HD}{\mathcal{H}_P}
\newcommand{\F}{\mathcal F}
\newcommand{\Folner}{F{\o}lner }

\newcommand{\M}{\mathcal{M}}

\newcommand{\N}{\mathbb{N}}
\newcommand{\Z}{\mathbb{Z}}

\newcommand{\und}{\underline}

\newcommand{\MG}{\mathcal{M}_G}

\newcommand{\MC}{\mathcal M(\mathscr{C})}
\newcommand{\MCe}{\mathcal M(\mathscr{C})^{e}}

\newcommand{\dbar}{\bar{d}}

\renewcommand{\phi}{\varphi}
\newcommand{\eps}{\varepsilon}
\makeatletter
\hypersetup{
    colorlinks=true,
    linkcolor=black,
    citecolor=blue,
    urlcolor=black,
    pdftitle={Overleaf Example},
    pdfpagemode=FullScreen,
    }
\author[Sejal Babel]{Sejal Babel}
\address[Sejal Babel]{
Faculty of Mathematics and Computer Science, Jagiellonian University in Krakow, ul. \L o\-jasiewicza 6, 30-348 Krak\'ow, Poland
}\email{sejal.babel@doctoral.uj.edu.pl}

\author[Martha Łącka]{Martha Łącka}
\address[Martha Łącka]{
Faculty of Mathematics and Computer Science, Jagiellonian University in Krakow, ul. \L o\-jasiewicza 6, 30-348 Krak\'ow, Poland
}\email{martha.ubik@uj.edu.pl }

\author[Marcel Mroczek]{Marcel Mroczek}
\address[Marcel Mroczek]{
Faculty of Mathematics and Computer Science, Jagiellonian University in Krakow, ul. \L o\-jasiewicza 6, 30-348 Krak\'ow, Poland
-- \and --
Doctoral School of Exact and Natural Sciences, Jagiellonian University, 
ul. \L o\-jasiewicza 11, 30-348 Krak\'ow, Poland
}\email{marcel.mroczek@doctoral.uj.edu.pl}

\theoremstyle{plain}
\newtheorem{thm}{Theorem}[section]
\newtheorem{lem}[thm]{Lemma}
\newtheorem{cor}[thm]{Corollary}

\theoremstyle{definition}
\newtheorem{defn}[thm]{Definition}
\newtheorem{rem}[thm]{Remark}
\newtheorem{ex}[thm]{Example}

\DeclareMathOperator{\Gen}{Gen}
\makeatletter
\def\blfootnote{\gdef\@thefnmark{}\@footnotetext}

\title[Generic Points in a Characteristic Class for Amenable Group Actions]{Generic Points in a Characteristic Class for Amenable Group Actions Are Closed in the Besicovitch Pseudometric}

\begin{document}
\begin{abstract}
We consider an action of a countable amenable group on a compact metric space, focusing on the set of generic points with respect to a fixed \Folner sequence. For a given characteristic class, we prove that the set of points that are generic (along the \Folner sequence) for some measure in this class is closed with respect to the Besicovitch pseudometric.
\end{abstract}

\blfootnote{\textup{2020} \textit{Mathematics Subject Classification}: 37B05, 37A15, 22F10}
\date{\today}
\maketitle

	\vspace{7mm}

	\textbf{keywords:} \emph{characteristic classes, the Besicovitch pseudometric, amenable groups, joinings, generic points}
		\vspace{7mm}
        
\section{Introduction}
The Besicovitch pseudometric has long been an important tool in the study of dynamical systems, providing a means of comparing orbits via their asymptotic average distance. Its definition captures a natural notion of proximity between the long-term behaviours of two points. It also admits a measure-theoretic counterpart: when considered on the space of invariant probability measures, it induces the metric $\bar\rho$, which generalises Ornstein’s celebrated $\bar d$ metric~\cite{Spec}. In this sense, the Besicovitch pseudometric serves as a bridge between pointwise and measure-theoretic perspectives, allowing techniques from one setting to inform the other.

Within the measure-theoretic framework, we consider characteristic classes that were introduced by Kanigowski, Kułaga-Przymus, Lemańczyk, and de la Rue~\cite{KKPLT23} as a collection of measure-preserving systems stable under natural operations: taking factors and forming countable joinings. Such classes arise in a variety of contexts, often corresponding to systems sharing a common structural property or rigidity feature. In~\cite{BL}, it was established that if one considers a sequence of ergodic measures on a topological dynamical system drawn from a given characteristic class, then any $\bar\rho$–limit of these measures remains within the class. This closure property reflects a certain robustness of the defining features of the class under metric convergence.

The results of~\cite{BL} were formulated in the classical setting of $\mathbb{Z}$–actions. However, many problems in ergodic theory and topological dynamics naturally arise in the broader context of countable amenable group actions. Every countable abelian group, including $\mathbb{Z}$, is amenable, and the amenability property guarantees the existence of Følner sequences. These sequences play a role analogous to intervals in $\mathbb{Z}$–actions, making it possible to take averages and extend a wide range of arguments from the $\mathbb{Z}$–case to this more general setting. In particular, several dynamically generating pseudometrics—including the Besicovitch pseudometric—admit natural analogues for actions of countable amenable groups, retaining many of their key properties.

The purpose of this note is to extend the main result of~\cite{BL} to this more general framework. Specifically, we show that for ergodic measures arising from an action of a countable amenable group on a compact metric space, the set of measures belonging to a given characteristic class is closed with respect to the $\bar\rho$ metric. This demonstrates that the stability phenomenon observed in the $\mathbb{Z}$–setting persists across the full breadth of countable amenable group actions.


\section{Preliminaries}
\subsection{Basic Notation}
Throughout this paper $X$ is a compact metric space with a metric $d$ and $G$ is a countable group acting on $X$ by homeomorphisms. 


We denote the Borel $\sigma$-algebra of $X$ by $\mathcal{B}(X)$. Let $\mathcal{M}(X)$ be the collection of all Borel probability measures on $X$. Given a dynamical system $(X, G)$, we denote by $\mathcal{M}_G(X)$ the set of all $G$-invariant measures in $\mathcal{M}(X)$, that is $\mu\in \MG(X)$ if and only if for every $A\in \mathcal{B}(X)$ and every $g\in G$ one has $\mu(A)=\mu(g A)$. It is known that if $G$ is amenable (see Subsection~\ref{amenable}), then $\M_G(X)$ is non-empty. 
A measure $\mu\in\M_G(X)$ is \emph{ergodic} if  every $G$-invariant measurable set $B$ satisfies $\mu(B)=0$ or $\mu(B)=1$. We denote by $\M_G^e(X)$ the set of all ergodic measures on $X$. Each $\mu\in \MG(X)$ defines a measure-preserving system $(X,G,\mu)$. 


For each $x \in X$, the symbol $\delta_x$ denotes the Dirac measure supported on $\{x\}$. 
It is well known that the weak$^*$ topology on $\M(X)$ can be metrised by the Prokhorov metric (see \cite{Prokhorov}). We denote by $\HD$ the Hausdorff distance induced by the Prokhorov metric on the family of non-empty and compact (with respect to the weak$^*$ topology) subsets of $\M(X)$.

For any $x \in X$ and $\varepsilon > 0$, the open ball centred at $x$ with radius $\varepsilon$ is defined as
$B(x, \varepsilon) := \{z \in X \,:\, d(z, x) < \varepsilon\}.$
For a subset $B \subset X$, its boundary is denoted by $\partial B$, and its \emph{$\varepsilon$-neighbourhood} is defined as
$B^{\varepsilon}:= \{x \in X \,:\, \inf_{b \in B} d(x, b) < \varepsilon\}.$

\subsection{Amenable Groups}\label{amenable}
Let $G$ be a countable discrete group.  
A sequence of non-empty finite subsets $\F=(F_n)_{n\in\N}$ of $G$ is called a \emph{(left) \Folner sequence} if, for every $g\in G$,
\[
\lim_{n\to\infty}\frac{\lvert g F_n \,\Delta\, F_n\rvert}{\lvert F_n\rvert}=0,
\]
where $A\Delta B$ denotes the symmetric difference of the sets $A$ and $B$.  
A group $G$ is said to be \emph{amenable} if it admits a (left) \Folner sequence.
A basic example is the additive group of integers $\mathbb{Z}$, where one may take $F_n=\{0,1,\ldots,n\}$ for each $n\in\N$. More generally, every countable abelian group is amenable.
A sequence $\F=(F_n)_{n\in\N}$ is called \emph{(right) \Folner} if
\[
\lim_{n\to\infty}\frac{\lvert F_n g \,\Delta\, F_n\rvert}{\lvert F_n\rvert}=0,
\quad\text{for every } g\in G.
\]
It is called \emph{two-sided \Folner} sequence if it is both (left) and (right) \Folner.
It was shown in~\cite[Theorem~1'']{symmetric} that if an amenable group is unimodular—which is always the case for discrete groups—then it admits a two-sided \Folner sequence. By contrast, in the general (non-unimodular) case, one can construct sequences that are (left) \Folner but fail to be (right) \Folner; see~\cite{bjorklund2018ergodic} for explicit examples.
In this note, unless otherwise specified, ``\Folner sequence'' will refer to a (left) \Folner sequence.


\subsection{Generic Points}
Consider a \Folner sequence $\F=(F_n)_{n\in\N}$ in the group $G$. A point $x\in X$ is said to be \emph{generic} for a measure $\mu\in \M(X)$ along $\F$ if 
\begin{eqnarray}\label{fff}\frac{1}{\lvert F_n \rvert}\sum_{f\in F_n}\delta_{f x}\xrightarrow[n\to \infty]{}\mu\quad \text{in the weak$^*$ topology.}\end{eqnarray}
It is immediate that whenever~\eqref{fff} holds, we have $\mu\in \MG(X)$. We denote by $\Gen_{\F}(\mu)$ the (possibly empty) set of all generic points for $\mu$ along $\F$.

If equation \eqref{fff} holds along a subsequence $(F_{k_n})_{n\in\N}$ of the \Folner sequence $\F$, we say that $x$ is \emph{quasigeneric} for the measure $\mu$ along $\F$. Given $x\in X$, we define $\hat{\omega}_{\F}(x)\subseteq \MG(X)$ to be the set of all measures $\mu$ such that $x$ is quasigeneric for $\mu$ along $\F$. By compactness of $\M(X)$ with respect to the weak$^*$ topology, we have $\hat\omega_{\F}(x)\neq\emptyset$ for every $x\in X$. 

We call a \Folner sequence $\mathcal F=(F_n)_{n\in\N}$ \emph{tempered} if there exists a positive constant $C$ satisfying 
\[\Big|\bigcup_{k\leq n}F_k^{-1}F_{n+1}\Big|\leq C\cdot\left|F_{n+1}\right| \text{ for every }n\in\N.\]
Lindenstrauss~\cite[Proposition 1.5.]{Lindenstrauss01} demonstrated that every \Folner sequence has a tempered subsequence. Therefore, every amenable group admits a tempered \Folner sequence. Moreover, it follows from the Lindenstrauss pointwise ergodic theorem that, if $\F$ is a tempered \Folner sequence, then for every $\mu\in\M_G^e(X)$, one has $\mu(\Gen_{\F}(\mu))=1$.

\subsection{Factors and Joinings}
Let $(X,G,\mu)$ and $(Y,G,\nu)$ be measure-preserving systems, with the same acting group $G$. We say that $(Y,G,\nu)$ is a \emph{factor} of $(X,G,\mu)$ if there exists a measurable map $\pi:X\to Y$ such that $\pi(g x)= g\pi(x)$, for $\mu$ almost every $x\in X$ and every $g\in G$, and $\pi^{*}(\mu)=\nu$. 
Consider a family of measure-preserving systems $(X_i,\mu_i,G)_{i\in I}$, where $I\subseteq\N$. A \emph{joining} of the sequence of measures $(\mu_i)_{i\in I}$ is a measure on the Cartesian product $\Pi_{i\in I}X_i$, whose marginals on $X_i$'s are $\mu_i$'s and that
is invariant under the product action of $G$. Let $J(\mu_1,\mu_2,\ldots) $ and  $J^e(\mu_1,\mu_2,\ldots)$ denote the set of all joinings and ergodic joinings of the sequence of measures $(\mu_i)_{i\in\N}$,  respectively. 
The following theorem states that, in order to determine whether there exists a factor map between two measure-preserving systems, it is enough to check the existence of a special joining (see \cite[Chapter 6]{Glas03}).

\begin{thm}\label{thm:Factor}
Let $G$ be a countable discrete group, and let $(X,G,\mu)$ and $(Y,G,\nu)$ be two measure-preserving systems. Suppose there exists a joining $\lambda \in J(\mu, \nu)$ such that
\[
\{X, \emptyset\} \otimes \mathcal{B}(Y) \subset \mathcal{B}(X) \otimes \{Y, \emptyset\}\pmod{\lambda}.
\]
Then  $(Y,G,\nu)$ is a factor of $(X,G,\mu)$.
\end{thm}
Here, the notation `$\mathcal{Z}_1\subset \mathcal{Z}_2\pmod{\lambda}$' means that for every $C\in \mathcal{Z}_1$ there exists $D\in \mathcal{Z}_2$ such that $\lambda(C\Delta D)=0$.

\subsection{Characteristic classes}
The notion of a characteristic class was introduced in~\cite{KKPLT23} for $\Z$-actions. Here, we consider it in the broader setting of group actions.
\begin{defn}\label{Characteristic class}
A collection of measure-preserving systems $\mathscr{C}$ is called a \emph{characteristic class} if it is closed under the operations of taking factors and countable joinings. 
\end{defn}
Let $\mathscr{C}$ be a characteristic class. Denote by $\MC\subseteq \M_G(X)$ the set of measures $\mu$ for which $(X,G,\mu)\in \mathscr{C}$, and by $\MCe$ the subset of ergodic measures in $\MC$.

Similarly to the case of $\Z$-actions (see \cite[p. 16]{KKPLT23}) characteristic classes form for example:
\begin{itemize}
\item the family of all measure preserving $G$-actions on standard Borel probability spaces,
\item zero entropy actions,
\item actions with discrete spectrum,
\item actions with rational discrete spectrum,
\item the class of factors of all self-joinings of a fixed dynamical system (note that this is the smallest characteristic class that contains a given measure preserving system).
\end{itemize}
From now on we assume that $G$ is an amenable countable discrete group, so that \Folner sequences exist.
\subsection{The Upper Asymptotic Density}
For a finite non-empty set $F\subset G$ and $A\subset G$ let
\[
D_F(A)=|A\cap F|/|F|
\]
The \emph{upper asymptotic density} of $A$ with respect to the \Folner sequence $\mathcal F=(F_n)_{n\in\N}$ is defined as
\[\overline{d}_{\mathcal F}(A)=\limsup_{N\to\infty}D_{F_N}(A).\]

\subsection{The Besicovitch Pseudometric}
Given a \Folner sequence $\F=(F_n)_{n\in\N}$, the \emph{Besicovitch pseudometric} on $X$ is defined by
	\[
D_{B,\mathcal F}(x,z)=\limsup_{N\to\infty}\frac{1}{|F_N|}\sum_{g\in F_N}d(gx,gz) \text{ for }x, z\in X.
	\]
We will make use of the following results stated without proofs. For proofs or further details, see~\cite[Proposition 1, Corollary 1]{DI},~\cite[Corollary~11, Theorem 17]{LS}, \cite[Theorem 4, Theorem 7]{KLO2}, and \cite[Theorem 4]{LM}.

Denote by $X^G$ the space of all functions from $G$ to $X$, namely 
\[
X^G=\{\underline{x}_G\mid \underline{x}_G=(x_g)_{g\in G},\; x_g\in X\}.
\]
\begin{thm}\label{lem:db-prim}
Fix $\mathcal F$ and define
$D'_{B,\F}$ on $X^G$ as
\[
D'_{B,\mathcal F}(\und{x}_G,\und{z}_G)=\inf
\big\{\delta>0\mid \dbar_{\mathcal F}(\{g\in G\mid d(x_g,z_g)\ge \delta\})<\delta\big\}.
\]
Then $D_{B,\mathcal F}$ and $D_{B,\mathcal F}'$ are uniformly equivalent on $X^G$. Moreover, the modulus of uniform equivalence does not depend on the choice of a \Folner sequence.
\end{thm}

\begin{thm}\label{HD}
For every $\eps>0$, there exists $\delta>0$ such that for all $x,x'\in X$ with $D_{B,\mathcal F}(x,
 x')<\delta$, one has $\HD(\hat\omega_{\F}(x),\hat\omega_{\F}(x'))<\eps$. Moreover, the number $\delta$ does not depend on the choice of a \Folner sequence $\mathcal F$.
\end{thm}

\begin{thm}\label{unequ}
If $\mathcal A$ is a finite set with discrete topology, then for every \Folner sequence $\mathcal F=(F_n)_{n\in\N}$ the Besicovitch pseudometric associated with any admissible metric on
the product space $\mathcal A^G$ endowed with the shift action of $G$ and the
\emph{$\bar{d}_{\mathcal F}$-pseudometric} on $\mathcal A^G$ given for $\underline{x}=(x_g)_{g\in G},\underline{z}=(z_g)_{g\in G}\in  \mathcal A^G$ by
\[
\bar{d}_{\mathcal F}(\underline{x},\underline{z})=\bar{d}_{\mathcal F}(\{g\in G \mid x_g\neq z_g\})=\limsup_{n\to\infty}\frac{\big|\{f\in F_n \mid x_f\neq z_f\}\big|}{|F_n|}
\]
are uniformly equivalent, and the modulus of uniform equivalence does not depend on $\mathcal F$.
\end{thm}

  
\begin{thm}\label{thm:Besicovitch-quasi-genericity-and-ergodicity}
Let $(\mu_p)_{p \in \N}\subseteq\mathcal M_G^e(X)$, $(x_p)_{p\in \N}\subseteq X^{\infty}$ and $x\in X$ be such that $D_{B,\F}(x_p,x)\to 0$ and for every $p\in\N$ one has $x_p\in\Gen_{\F}(\mu_p)$. Then there exists $\mu\in\M_G^e(X)$ such that $x \in \Gen_{\F}(\mu)$.
\end{thm}


\section{The $\bar\rho$ Metric on the Space of $G$-invariant Measures}\label{Secmetric}
In this section, we prove that the Besicovitch pseudometric induces a metric on the space of $G$-invariant measures.
Our approach begins with an alternative definition of the distance between two measures, formulated so as to be independent of the choice of a \Folner sequence. This leads naturally to the function $\bar\rho$ (see Definition~\ref{dddd}). For $\mathbb{Z}$–actions, a proof that $\bar\rho$ is a metric is included in the thesis~\cite{Babel}, and the argument is similar for $G$-actions; we present it here in full for completeness. The properties of $\bar\rho$ have been examined in detail in~\cite{Spec} for $\Z$-actions and in~\cite{KLM} for amenable group actions.

\begin{defn}\label{dddd}
For $\mu,\nu \in \M_G(X)$ define
\[
\bar\rho(\mu,\nu):=\inf_{\lambda\in J(\mu,\nu)}\int_{X\times X} d\;\text{d}\lambda.
\]
\end{defn}

\begin{rem}
In Definition~\ref{dddd}, the infimum is in fact a minimum, since 
$J(\mu,\nu)$ is weak$^*$ compact and the functional
\[
\lambda \,\longmapsto\, \int_{X\times X} \rho(a,b)\, d\lambda(a,b),
\quad \lambda \in \mathcal{M}_G(X\times X),
\]
is weak$^*$ continuous. Moreover, the minimum can be achieved by an ergodic joining.
\end{rem}
\begin{rem}
It is evident that $\bar\rho$ is a symmetric function from $\M_G(X)\times\M_G(X)$ to $\mathbb R_+$. 
\end{rem}
\begin{lem}\label{ozn}
	One has $\bar\rho(\mu, \nu)=0$ if and only if $\mu=\nu$.
\end{lem}
\begin{proof}
	The ``if'' part is obvious. To demonstrate the ``only if'' part, note that if $\bar\rho(\mu,\nu)=0$, then there exists a joining of $\mu$ and $\nu$ which is supported on the diagonal (that is on the set $\{(x,z)\in X\times X\,:\,x=z\}$).
\end{proof}

In the proof of Lemma~\ref{triangle}, we mimic the reasoning from \cite{Glas03,Villani}.

\begin{lem}\label{triangle}
The function $\bar\rho\colon\M_G(X)\times\M_G(X)\to\mathbb R _+$ satisfies the triangle inequality.
\end{lem}
\begin{proof}
Let $\mu, \nu, \eta\in\M_G(X)$. Assume that $\pi_{1,2}\in J(\mu,\eta), \pi_{2,3}\in J(\eta, \nu)$ are such that 
\[
\bar\rho(\mu,\eta)=\int d(x,y)\,\text{d}\pi_{1,2}(x,y)\quad\text{ and }\quad\bar\rho(\eta,\nu)=\int d(y,z)\,\text{d}\pi_{2,3}(y,z).
\]
Let $\pi^{1,2}\colon y\mapsto\pi_y^{1,2}$ and $\pi^{2,3}\colon y\mapsto\pi_y^{2,3}$ be respectively disintegrations of $\pi_{1,2}$ and $\pi_{2,3}$ over $\eta$ (using factor maps $\text{proj}_{1,2}\colon X\times X\to X$, $\text{proj}_{2,3}\colon X\times X\to X$).
Define $\zeta\in \M(X\times X\times X)$ by \[\zeta=\int_X\pi_y^{1,2}\otimes\delta_y\otimes\pi_y^{2,3}\,\text{d}\eta(y).\]
Then ${\text{proj}_{1,2}}_*(\zeta)=\pi_{1,2}$ and
${\text{proj}_{2,3}}_*(\zeta)=\pi_{2,3}$.
Let $\pi_{1,3}={\text{proj}_{1,3}}_*(\zeta)$\footnote[2]{$\pi^{1,3}$ is known as \emph{relatively independent joining} of $\pi^{1,2}$ and $\pi^{1,3}$ over $\eta$ (see \cite[p.~126]{Glas03}).}. Then $\pi_{1,3}\in J(\mu,\nu)$ and
\begin{multline*}
\bar\rho(\mu,\nu)\leq\int_{X\times X}d(x,z)\,\text{d}\pi_{1,3}(x,z)=\int_{X\times X\times X}d(x,z)\,\text{d}\zeta (x,y,z) \\\leq\int_{X\times X\times X}d(x,y)+d(y,z)\,\text{d}\zeta (x,y,z)\\=\int_{X\times X}d(x,y)\,\text{d}\pi_{1,2}(x,y)+\int_{X\times X}d(y,z)\,\text{d}\pi_{2,3}(y,z)=\bar\rho(\mu, \eta)+\bar\rho(\eta, \nu).\qedhere
\end{multline*} 
\end{proof}
\begin{cor}
The function $\bar\rho$ is a metric on $\M_G(X)$.
\end{cor}

We introduce the function $\overline{D}_{B,\F}$ on the space of $G$-invariant ergodic measures, obtained as the metric induced by the Besicovitch pseudometric (see Definition~\ref{metr}).
The construction is based on two-sided tempered Følner sequence. We prove that the resulting function is indeed a metric on $\M_G^e(X)$ and coincides with $\bar\rho$.
This identification allows us to reduce the study of the system to the dynamics of a single orbit.
The argument follows the approach of~\cite{Shields}. 

The following Lemma~\ref{lem:inJ} is straightforward, and we omit its proof.

\begin{lem}\label{lem:inJ}
  Let $\F$ be a \Folner sequence. If $\mu,\nu\in\M_G^e(X)$, $x\in \Gen_{\F}(\mu)$, $z\in\Gen_{\F}(\nu)$ and $\lambda\in\hat\omega_{\F}{(x,z)}$, then $\lambda\in J(\mu,\nu)$.
  \end{lem}

\begin{rem}
Note that the lemma above is meaningful only when the sets $\Gen_{\F}(\mu)$ and $\Gen_{\F}(\nu)$ are non-empty. This is always the case when $\mu,\nu\in\M_G^e(X)$ and $\F$ is tempered; however for given $\mu$ and $\nu$, it may also hold for other \Folner sequences.
  \end{rem}

\begin{defn}
Let $\nu\in\M_G^e(X)$, let $\F$ be a two-sided tempered \Folner sequence, and let $x\in X$. Define
\[	
\overline{D}_{B,\F}(x,\nu):=\inf\big\{D_{B,\F}(x,z)\,\mid\,z\in\Gen_{\F}(\nu)\big\}.
\]
\end{defn}

\begin{lem}\label{lem:constant}
Let $\mu,\nu\in\M_G^e(X)$ and let $\F$ be a two-sided \Folner sequence. Then the function 
\[
\Gen_{\F}(\mu)\ni x\mapsto \overline{D}_{B,\F}(x,\nu)\in\mathbb R_+\]
is constant $\mu$-almost everywhere.
\end{lem}
\begin{proof}
Since the \Folner sequence is two-sided, we have that $\overline{D}_{B,\F}(gx,\nu)=\overline{D}_{B,\F}(x,\nu)$ for all $g\in G$ and $x\in X$. In other words, the function $X\ni x\to \overline{D}_{B,\F}(x,\nu)$ is $G$-invariant. Therefore the claim follows from the ergodicity of $\mu$.
\end{proof}
\begin{defn}\label{metr}
For a two-sided tempered \Folner sequence $\F$ and $\mu,\nu\in\M_G^e(X)$, we define $\overline{D}_{B,\F}(\mu,\nu)$ to be the constant value obtained in Lemma~\ref{lem:constant}.
\end{defn}
In Theorem~\ref{nowy}, we show that $\bar\rho$ and $\overline{D}_{B, \F}$ coincide on the set $\M_G^e(X)$. In particular, this implies that $\overline{D}_{B, \F}$ is a metric on $\M_G^e(X)$ and does not depend on the choice of a two-sided tempered \Folner sequence.
\begin{thm}\label{nowy}
 Let $\mu, \nu\in\M_G^e(X)$. Fix a \Folner sequence $\F$ (not necessarily two-sided, nor tempered). Then for all $x\in\Gen_{\F}(\mu)$ and $z\in\Gen_{\F}(\nu)$ we have  $D_{B,\F}(x,z)\geq \bar\rho(\mu,\nu)$.
\end{thm}
\begin{proof}
For every $n\in\N$ let \[\mathbf{m}\big((x,z), F_n\big)=\frac{1}{|F_n|}\sum_{f\in F_n}\delta_{(fx, fz)}.\]
Note that
\begin{equation*}
 D_{B,\F}(x,z)=\limsup\limits_{N\to\infty}\frac{1}{|F_N|}\sum_{f\in F_N}d(fx,fz)=\limsup\limits_{N\to\infty}\int d\,\text{d}\mathbf m\big((x,z), F_n\big).
 \end{equation*}
Therefore the claim follows from Lemma~\ref{lem:inJ}.
\end{proof}
\begin{rem}
Note that in fact we can strengthen Theorem~\ref{nowy} by replacing the $D_{B,\F}(x,z)$ by \[\liminf\limits_{N\to\infty}\frac{1}{|F_N|}\sum_{f\in F_N}d(fx, fz)\]
(see also \cite{Glas03}).
 \end{rem}
\begin{thm}\label{follows}
  For all $\mu,\nu\in\M_G^e(X)$ and any tempered two-sided \Folner sequence $\F=(F_n)_{n\in\N}$ one has $\overline{D}_{B,\F}(\mu,\nu)=\bar\rho(\mu,\nu)$.
\end{thm}
\begin{proof}
 It follows from Theorem~\ref{nowy} that $\overline{D}_{B,\F}(\mu,\nu)\geq \bar\rho(\mu,\nu)$.
 We will show the opposite inequality. Let $\lambda\in  J(\mu,\nu)$ be an ergodic joining such that \[\bar\rho(\mu,\nu)=\int_{X\times X}d\,\text{d}\lambda.\] Using the  Birkhoff ergodic theorem we obtain that for $\lambda$-almost all pairs $(x,z)\in X\times X$ the following two conditions are satisfied:
 \begin{enumerate}
\item $\displaystyle\frac{1}{|F_n|}\sum_{f\in F_n}d(fx,fy)\to \bar\rho(\mu,\nu)\text{ as }n\to\infty,$
 \item $(x,y)\in\Gen_{\F}(\lambda)\subseteq\Gen_{\F}(\mu)\times\Gen_{\F}(\nu)$.
\end{enumerate}
Thus $\overline{D}_{B,\F}(\mu,\nu)\leq \bar\rho(\mu,\nu)$.
 \end{proof}

The following corollary extends~\cite[Corollary 7.3.]{Spec} where the result was established for $\Z$-actions. 
 \begin{cor}
If $\mu,\nu\in\M_G^e(X)$ satisfy $\bar\rho(\mu,\nu)<
\eps$ for some $\eps>0$ and $\F$ is a tempered two-sided \Folner sequence, then there exist $x,z\in X$ such that $x$ is generic for $\mu$ along $\F$, $z$ is generic for $\nu$ along $\F$ and $D_{B,\F}(x,z)<\eps$.
\end{cor}

\begin{proof}
By Theorem~\ref{follows},
$\bar\rho(\mu,\nu)=\overline{D}_{B,\F}(\mu,\nu)$.
Hence, there exist points $x,z\in X$ such that $x$ is generic for $\mu$ along $\F$ and $z$ is generic for $\nu$ along $\F$ and
$D_{B,\F}(x,z)<\eps$, as required.
\end{proof}

In fact repeating the proof from \cite{Spec} allows to skip the assumption that the \Folner sequence is two-sided.


\section{Main results}

The following lemma is folklore (see also \cite[Lemma 3.1.]{bogachev07}).
\begin{lem}\label{lem:Generating_partition} 
Let $(\nu_k)_{k=1}^{\infty}\subseteq \M(X)$. Then there exists a countable basis $\beta$ for the topology of $(X,d)$ such that for every $B\in\beta$ one has
$\nu_k(\partial B)=0 \text { for every } k\in \N$.
\end{lem}

Theorem~\ref{joining_thm} extends \cite[Lemma 3.17]{BKLR19} and \cite[Theorem 3.3.]{BL}.

\begin{thm}\label{joining_thm}
Let $(X,G)$ be a topological dynamical system. Suppose that there exist sequences $(x_m)_{m=1}^{\infty}\subseteq X$, $(\mu_m)_{m=1}^{\infty}\subseteq \M_G(X)$, a point $x\in X$, and a \Folner sequence $\F=(F_k)_{k=1}^{\infty}$ in $G$ such that 
 \begin{enumerate}[(i)]
\item for every $m\in\N$ the point $x_m$ is generic for $\mu_m$ along $\F$,
\item $D_{B,\F}(x_m,x)\to 0$ as $m\to \infty$.
\end{enumerate}
Then the point $x\in X$ is generic for some measure $\mu\in\M_G(X)$ along $\F$, and there exists a joining $\nu\in J(\mu_1,\mu_2,....)$ such that $(X,G,\mu)$ is a factor of $(X^{\infty},G,\nu)$. Furthermore, if each $\mu_m$ is ergodic, then $\nu$ may be chosen ergodic.
\end{thm}
\begin{proof}
It follows from Theorem~\ref{HD} that $x$ is generic for some $\mu \in \mathcal{M}_G(X)$ along $\mathcal{F}$. To construct a joining with the required properties, consider  
\[
\hat{x} := (x, x_1, x_2, \ldots) \in X \times X^{\infty},
\]
and define $(\xi_k)_{k\geq 1} \subseteq \mathcal{M}(X \times X^\infty)$ by setting
\[
\xi_k := \frac{1}{|F_k|} \sum_{f \in F_k} \delta_{f(\hat{x})}, 
\quad k \geq 1.
\]
Let $\bar{\nu} \in \mathcal{M}_{G}(X \times X^{\infty})$ be a weak$^*$ limit point 
of $(\xi_k)_{k\geq 1}$ (along a subsequence), which exists by the compactness of $\mathcal{M}(X \times X^\infty)$. Since $x$ is generic for $\mu$ along $\mathcal{F}$ and for every $m \geq 1$,
$x_m$ is generic for $\mu_m$ along $\mathcal{F}$, we have 
$\bar{\nu} \in J(\mu, \mu_1, \mu_2, \ldots)$. 

By Theorem~\ref{thm:Factor}, to show that $(X,G,\mu)$ is a factor of $(X^\infty,G,\nu)$, where $\nu$ is the marginal of $\bar{\nu}$ on $X^{\infty}$, it suffices to show that
\begin{equation} \label{eq:sigma_algebra}
\mathcal{B}(X) \otimes \{X^\infty,\emptyset\} 
\subset \{X,\emptyset\} \otimes \mathcal{B}(X^\infty) 
\pmod{\bar{\nu}}.
\end{equation}
Using Lemma~\ref{lem:Generating_partition}, choose a basis $\beta$ generating $\mathcal{B}(X)$ such that, for every $B \in \beta$ and every $m \in \mathbb{N}$, we have $\mu(\partial B) = \mu_m(\partial B) = 0$. To establish \eqref{eq:sigma_algebra} it is enough to show that for every $B \in \beta$ there exists $A \in \mathcal{B}(X^\infty)$ such that
\[
\bar{\nu}\big((B \times X^\infty) \,\Delta\, (X \times A)\big) = 0.
\]
Fix $B \in \beta$. For $n \geq 1$, define
\[
A_n := 
(\underbrace{X \times \cdots \times X}_{n-1\ \text{times}} 
\times B \times X \times \cdots) \in \mathcal{B}(X^{\infty}).
\]
Since $\partial(\partial B) \subseteq \partial B$, we have $\mu(\partial(\partial B)) = 0$. By the portmanteau theorem,
\[
\frac{1}{|F_k|} \sum_{f \in F_k} \delta_{fx}(\partial B) 
\;\longrightarrow\; \mu(\partial B) = 0
\quad\text{as } k \to \infty.
\]
Fix $\varepsilon > 0$. There exists $\delta < \varepsilon/2$ such that
\[
\limsup_{k \to \infty} \frac{1}{|F_k|} 
\big\lvert \{ f \in F_k : fx \in (\partial B)^{\delta} \} \big\rvert 
< \frac{\varepsilon}{2}.
\]
By assumption, there exists $N \in \mathbb{N}$ such that 
$D'_{B,\mathcal{F}}(x, x_N) < \delta$, that is.,
\[
\bar{d}_{\mathcal{F}}\big( \{ g \in G : d(gx, gx_N) \geq \delta \} \big) 
< \delta.
\]
Let $\pi_N : X \times X^\infty \to X \times X$ denote the projection
\[
\pi_N(y, y_1, y_2, \ldots, y_N, \ldots) := (y, y_N),
\]
and let $\nu_N$ be the push-forward of $\bar{\nu}$ under $\pi_N$. Then
\begin{multline*}
\bar{\nu}\big((B\times X^\infty) \,\Delta\, (X\times A_{N})\big) 
= \nu_N\big((B\times X) \,\Delta\, (X\times B)\big) \\
= \lim_{k\to\infty} \frac{1}{|F_k|} 
\big\lvert \{ f \in F_k : (fx, fx_N) \in (B\times B^c) \cup (B^c\times B) \} 
\big\rvert \\
\leq \limsup_{k\to\infty} \frac{1}{|F_k|} 
\big\lvert \{ f \in F_k : d(fx, fx_N) \geq \delta \} \big\rvert \\
\quad + \limsup_{k\to\infty} \frac{1}{|F_k|} 
\big\lvert \{ f \in F_k : fx \in (\partial B)^{\delta} \} \big\rvert 
< \delta + \frac{\varepsilon}{2} < \varepsilon.
\end{multline*}
Thus, for every $\eps>0$ there exists $N\in\N$ so that for all $n\geq N$ we obtain 
\begin{equation}\label{eq:joining_ii}
\bar{\nu}((B\times X^{\infty})\Delta (X\times A_n))<\eps.
\end{equation}
Define  
\[
\hat\rho(C,D)=\bar{\nu}(C\Delta D),\quad  C,D\in \mathcal{B}(X\times X^\infty). 
\]
It is well known that, after identifying sets that differ by $\bar{\nu}$-null set, the space $(\mathcal{B}(X\times X^\infty),\hat{\rho})$ is a complete metric space (refer to \cite[Theorem 1.12.6]{bogachev07} for details).

From \eqref{eq:joining_ii}, it follows that $(X\times A_n)_{n=1}^{\infty}$ is Cauchy in $(\mathcal{B}(X\times X^\infty),\hat{\rho})$ and thus converges to some $C\in \mathcal{B}(X\times X_\infty)$. Since the family 
\[
\mathcal{A}=\{X\times A:A\in \mathcal{B}(X^\infty)\}
\]
is closed in metric $\hat{\rho}$, the limit set $C$ must also belong to $\mathcal{A}$. Hence there exists $A\in \mathcal{B}(X^\infty)$ such that $C=X\times A$, and hence $\bar{\nu}((B\times X^\infty) \Delta (X\times A))=0$. This establishes that $(X,G,\mu)$ is a factor of $(X^\infty, G, \nu)$.
 

For the ``furthermore'' part of the theorem, let 
$\mathcal{C} \subset \mathcal{B}(X^\infty)$ be a $G$-invariant 
sub-$\sigma$-algebra representing the factor system $(X,G,\mu)$,
and let $\Lambda$ denote the ergodic decomposition of $\nu$. Restricting $\Lambda$ to $\mathcal{C}$ yields
\begin{equation} \label{eq:restriction}
    \mu \;=\; \nu_{\lvert \mathcal{C}} 
    \;=\; \int \lambda_{\lvert \mathcal{C}} \, d\Lambda(\lambda).
\end{equation}
By Theorem~\ref{thm:Besicovitch-quasi-genericity-and-ergodicity},
the measure $\mu$ is ergodic and the uniqueness of the ergodic decomposition implies
\begin{equation} \label{eq:lambda-equals-mu}
    \lambda_{\lvert \mathcal{C}} = \mu 
    \quad \text{for } \Lambda\text{-a.e. } \lambda.
\end{equation}
Thus, $(X,G,\mu)$ arises as a factor of 
$\Lambda$-almost every ergodic component of $(X^\infty,G,\nu)$. Since each $\mu_m$ with $m \geq 1$ is ergodic, every such component lies 
in $J^e(\mu_1,\mu_2,\dots)$. It follows that, for the purposes of the ``furthermore'' statement, the measure $\nu$ may be replaced by any of these components.
\end{proof}

 \begin{cor}
Let $\F$ be a \Folner sequence, and let $\mathscr{C}$ be a characteristic class. Define
\[
\Gen_{\F}(\mathscr{C})=\big\{x\in \Gen_{\F}(\mu) : \mu\in\MC\big\}
\]
to be the collection of generic points along $\F$ for measures in $\mathscr{C}$. Then $\Gen_{\F}(\mathscr{C})$ is closed with respect to $D_{B,\F}$.
\end{cor}

\begin{cor}
    The set  $\MCe$ is closed with respect to $\bar\rho$.
\end{cor}
\section{Applications}

In \cite{LM} the authors employed convergence in the Besicovitch pseudometric to establish ergodicity of certain dynamical systems by approximating generic points with periodic ones. These examples fit naturally within the scope of our theorem and can be revisited here as illustrative applications. However, instead of deducing ergodicity, we will conclude that the associated systems have zero entropy and discrete spectrum. For the reader’s convenience, we briefly recall these examples below.

\begin{ex}
A countable group $G$ is called \emph{residually finite} if there exists a descending sequence $(H_n)_{n\ge 0}$ of finite-index normal subgroups such that
\[
\bigcap_{n=0}^\infty H_n = \{e\}.
\]
If $G$ is amenable and residually finite, a result of Cortez--Petite~\cite[Lemma~4]{Cortez} yields such a sequence $(H_n)_{n\geq 0}$ together with a \Folner sequence $(F_n)_{n\ge 0}$ satisfying:
\begin{enumerate}
    \item[(i)]  $F_n$ is a fundamental domain for $G/H_n$ for every $n\in\N$,;
    \item[(ii)] for every $n \ge 0$, one has
    \[
    F_{n+1} = \bigsqcup_{v \in F_{n+1} \cap H_n} F_n v.
    \]
\end{enumerate}
Let $X = \{0,1\}^G$ be the full shift. Choose $\F=(F_n)_{n\in\N}, (H_n)_{n\in\N}$ as above with $[H_n : H_{n+1}] > 2^n$ for all $n$. Define $x^{(1)}$ to be the constant-$0$ configuration. Given $x^{(k)}$, obtain $x^{(k+1)}$ by tiling $F_{k+1}$ with translates of $F_k$, copying $x^{(k)}$ on all but one tile, placing its bitwise complement on the remaining tile, and extending periodically via $H_{k+1}$. Then $(x^{(n)})_{n\in\N}$ converges in $D_{B,\mathcal{F}}$ to a configuration $x$ which is generic for some zero entropy measure with discrete spectrum.
\end{ex}

 \begin{figure}[h]
    \centering
    \begin{tikzpicture}[scale=0.2]
        \foreach \y in {-25,...,25} {
            \foreach \x in {-25,...,25} {
                \pgfmathparse{int(abs(\x)+abs(\y) > 0 ? gcd(abs(\x), abs(\y)) : -1)}
                \let\gcdval\pgfmathresult

                \ifnum\gcdval=1
                    \fill[blue!80] (\x,\y) circle (6pt); 
                \else
                    \ifnum\gcdval>0
                        \fill[gray!60] (\x,\y) circle (4pt); 
                    \fi
                \fi
            }
        }
    \end{tikzpicture}
    \caption{Central region of the lattice with points visible from the origin marked as larger blue dots.}
    \label{fig:coprime_lattice_view}
\end{figure}
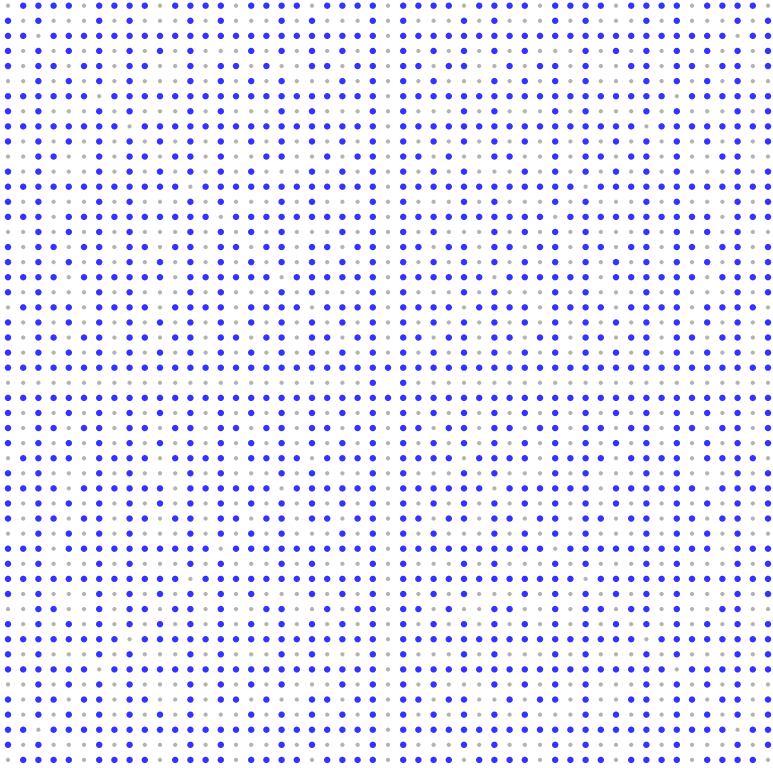

\begin{ex}\label{visible}

Let $(p_i)_{i\in\N}$ denote the increasing sequence of all prime numbers. Consider the set of \emph{visible points} in $\mathbb{Z}^2$ (see also Figure~\ref{fig:coprime_lattice_view} and \cite{Baake2015ErgodicPO,BfreeDynamics,KLW}), given by
\[
V := \left\{ (m,n) \in \mathbb{Z}^2 : \gcd(m,n) = 1 \right\} = \mathbb{Z}^2 \setminus \bigcup_i p_i \mathbb{Z}^2.
\]
Define $v \in \{0,1\}^{\mathbb{Z}^2}$ as the indicator function of $V$. To analyze the structure of $v$, we construct a sequence of approximating periodic configurations $(x^{(n)})_{n \in \N}\subset \{0,1\}^{\mathbb{Z}^2}$, where each $x^{(n)}$ is defined as the characteristic function of \[\mathbb{Z}^2 \setminus \bigcup_{i \leq n} p_i \mathbb{Z}^2.\]
Each $x^{(n)}$ exhibits periodicity with respect to the lattice $\left(\prod_{i=0}^{n} p_i\right)\mathbb{Z}^2$, and the discrepancy between $v$ and $x^{(n)}$ is localized within the union $\bigcup_{i > n} p_i \mathbb{Z}^2$. By applying Theorem~\ref{unequ}, we deduce that the sequence $(x^{(n)})_{n\in\N}$ converges to $v$ in the Besicovitch pseudometric (and so $v$ is generic for a zero entropy measure with discrete spectrum).
\end{ex}
\begin{rem}
    Example~\ref{visible} belongs to a general family of examples to which this technique extends \cite{MR4870827,Baake2015ErgodicPO,Kfree,MR4251829,MR4613803}, see~\cite[Example~9]{LM} for details.
\end{rem}

\section*{Acknowledgments}
The authors would like to thank Dominik Kwietniak for his useful suggestions and remarks. SB was supported by NCN Sonata Bis grant no. 2019/34/E/ST1/00237. MM acknowledges support of the National Science Centre (NCN), Poland, grant no. 2022/47/B/ST1/02866.

\bibliographystyle{plain}
\bibliography{library.bib}

\end{document}